\documentclass[12pt]{siamart190516}

\usepackage{mathtools}
\usepackage{amsmath}
\usepackage{amssymb}
\usepackage{stmaryrd}

\usepackage{bbm}
\usepackage{booktabs}

\newcommand{\id}{\mathbb{I}}
\newcommand{\supp}{\mbox{\rm supp}\; }
\newcommand{\grad}{\nabla{}}
\renewcommand{\div}{\mbox{\rm div}}

\newtheorem{remark}{Remark}

\newcommand{\e}{\varepsilon}

\title{Small dispersion approximation of shock wave dynamics}

\author{Misha Perepelitsa \thanks{Department of Mathematics, University of Houston, 4800 Calhoun Road, Houston, TX
		(\email{mperepel@central.uh.edu})}}

\begin{document}

\maketitle

\begin{abstract}
We introduce a dispersion approximation of weak, entropy solutions of multidimensional scalar conservation laws using 
variational kinetic representation, where equilibrium densities satisfy  Gibb's entropy minimization principle for a piecewise linear, convex entropy. For such solutions, we show that  small scale discontinuities, measured by the entropy increments, propagate with characteristic velocities, while the large scale, shock-type discontinuities propagate with speeds  close to the speeds of classical shock waves. In the zero-limit of the scale parameter, approximate solutions converge  to a unique, entropy solution of a scalar conservation law.
\end{abstract}

\begin{keywords}
Shock waves, entropy solutions, kinetic equations
\end{keywords}
\begin{AMS}
35L60, 35L65
\end{AMS}

\section{Introduction} Consider the Cauchy problem for a quasilinear system 
\begin{equation}
\label{QS}
\begin{cases}
\partial_t U{}+{}\sum_{i=1}^d\partial_{x_i} F_i(U){}={}0,\quad (x,t)\in\mathbb{R}^{d+1}_+,\\
U(x,0)=U_0(x),\quad x\in\mathbb{R}^d,
\end{cases}
\end{equation}
where $U:\mathbb{R}^{d+1}_+\to\mathbb{R}^m, $  $F_i:\mathbb{R}^m\to\mathbb{R}^{m}.$ The main difficulty in constructing weak solutions for quasilinear systems \eqref{QS} is the lack of apriori estimates on solutions in norms that control oscillations. This limits the application of such methods as  viscosity or relaxation approximations of \eqref{QS} for which pointwise convergence of approximate solutions is hard to establish.

The difficulty is well illustrated on an example of a shock wave.
For systems with a convex entropy, weak solutions are typically restricted to verify entropy dissipation balance:
\[
\partial_t\eta(U)+\div_x\, q(U){}={}r,\quad r{}\leq 0,
\]
which provides  apriori estimate on the total entropy at time $t$  and total dissipated entropy up to time $t$ in terms of the entropy of the initial data. This type of control is however too weak. For example, for a shock wave contained inside an interval $[a,b],$ the total dissipated entropy  $\int_0^t\int_a^b r\,dx$ is cubic in the strength of the shock,  see theorem 8.5.1 of Dafermos \cite{Dafermos}. Thus, in a regime of increasing number of small shock waves, the entropy does not control the oscillations as measured by sum of all shock wave strength.

In this  paper we explore the possibility of constructing approximate solutions of \eqref{QS}
for which  an entropy inequality implies strong compactness, at the price of distorting certain small scale details of the original solutions. More specifically, we will seek approximate, weak  solutions of \eqref{QS} with  large shocks  propagating with speeds close to the speeds computed from the original system \eqref{QS}, and  discontinuities, for which the change in the entropy is smaller than a certain threshold value $\e,$  are transported with charachteristic velocities. Thus, the approximation involves small scale dispersion effects. 

In this paper, we'll show how this type of approximation can be implemented  for scalar conservation laws in multi-dimensions:
\begin{eqnarray}
\label{SCL}
\partial_t \rho{}+{}\div_x A(\rho){}={}0,&\quad& (x,t)\in\mathbb{R}^{d+1}_+,
\end{eqnarray}

Our approach is based on the variational kinetic representation of entropy weak solutions of \eqref{SCL} developed by Brenier \cite{Brenier1, Brenier2}, Brenier and Corrias \cite{BrenierCorrias}, Giga and Miyakawa \cite{GM}, and Lions et al. \cite{LPT}. According to the theory,  an admissible, non-negative solution $\rho(x,t)$ is represented as a moment of an ``equilibrium'' kinetic density $f_{eq}:$ 
\begin{equation}
\label{Eq_density}
\rho(x,t){}={}\int f_{eq}(x,t,v)\,dv,\quad f_{eq}(x,t,v){}={}
\id_{[0,\rho(x,t)]}(v),
\end{equation}
with $f_{eq}$ solving  a kinetic equation
\begin{equation}
\label{kinetic_eq_1}
\partial_t f{}+{}A'(v)\cdot\grad_x f{}={}\partial_v m,
\end{equation}
where $m$ is non-negative Radon measure on $\mathbb{R}^{d+2}_+.$ Conversely, any solution of \eqref{kinetic_eq_1} constrained by condition \eqref{Eq_density} for some $\rho(x,t)$ defines an admissible weak solution of conservation law in \eqref{SCL}, see \cite{LPT}. 
Moreover, for any strictly convex function $\eta,$ and a.e. $(x,t),$ $f_{eq}(x,t,v)$ is the unique minimizer of the problem
\begin{equation}
\label{min:intro}
\min\left\{ \int \eta'(v)\tilde{f}(v)\,dv\,:\, \tilde{f}(v)\in[0,1],\,\int \tilde{f}\,dv{}={}\rho(x,t)\right\}.
\end{equation}
Solutions of \eqref{kinetic_eq_1} can be obtained as limits of solutions of a relaxation problem 
\begin{equation}
\label{BGK:intro}
\partial_t f{}+{}A'(v)\cdot\grad_x f{}={}h^{-1}(M_f-f),
\end{equation}
where $M_f$ is the minimizer of \eqref{min:intro} with $\rho=\int f\,dv.$

The approximate solutions, with the properties described above, will be obtained from the same variational kinetic formulation \eqref{min:intro} and \eqref{BGK:intro}, in which a strictly convex function $\eta$ is replaced by a continuous, piecewise linear approximate entropy $\eta_\e.$ With the new entropy function, the minimization problem admits multiple solutions, with indeterminacy on  small $\e$--scales. A particular minimizer $M_f$ will be selected so that $L^1$ norm of $f-M_f$ can be estimated by the entropy increment $\int\eta_\e'(v)(f-M_f)\,dv.$

Our main result, theorem \ref{th:1} describes the  kinetic functions  obtained from this kinetic relaxation approach.
Such kinetic functions  verify equation \eqref{kinetic_eq_1} where, in addition, the  right-hand side  is a signed Radon measure, with the total variation controlled by a single entropy:
 \[
||\partial_v m||{}\leq{}\frac{2}{\e}\int \eta_{\e}(\rho_0)\,dx.
\] 
Furthermore, we show that moments $\rho=\int f\,dv,$ and $\phi=\int A'(v)f\,dv,$ solve the  balance equation
\[
\partial_t\rho+\div_x\phi{}={}0,
\]
and $\phi(x,t){}={}A(\rho(x,t))+O(\e^2).$ In particular, if there is a co-dimension one discontinuity of $\rho$ with values $\rho^+,\rho^-,$ (such discontinuities do develop in the solutions), such that $|\rho^+-\rho^-|>\e,$  then it propagates with the velocity
\[
\sigma{}={}\frac{A(\rho^+)-A(\rho^-)}{\rho^+-\rho^-} {}+{}O(\e).
\]
The kinetic function $f,$ as well as its moments, depend in the scale parameter $\e.$ In theorem \ref{th:2} we show that in the limit of $\e\to0,$ $\rho=\rho^\e(x,t)$  converges to an admissible solutions of \eqref{SCL}.

In summary, we describe a new type of approximation of scalar conservation laws with properties distinct from the well-known viscosity approximation of Kruzhkov \cite{Kruzhkov}, kinetic relaxation approximation of Brenier \cite{Brenier1} and Giga and Miyakawa \cite{GM}, or semi-linear relaxation of Katsoulakis and Tzavaras \cite{KTz}. The model approximates the dynamics of large shock waves and controls small scale oscillations, by means of entropy balance.

\section{Main result}
Let $A\in C^2(\mathbb{R})^d.$ Without the loss of generality, we will assume that $\rho_0$ is non-negative and bounded, so that all kinetic functions are defined for the range of the kinetic variable $v\in [0,L],$ for some $L>0.$ Let $\e>0.$ Define a piecewise constant function $\eta_\e$ as
\[
\eta_\e(v){}={}k,\quad v\in[k\e,(k+1)\e),\,k=0..\lceil L/\e\rceil.
\]
$\eta_\e$ approximates the derivative of the quadratic entropy function. Here, for notational convenience, we use $\eta_\e$ to denote the derivative of the entropy function described in the  introduction.
\begin{theorem} 
\label{th:1}
Let $f_0\in L^1(\mathbb{R}^d\times[0,L])$ with values $\{0,1\}.$   For any $\e>0$ there is $f\in L^1(\mathbb{R}^{d+1}_+\times[0,L])$ with values in $[0,1]$ and $m$ -- a non-negative Radon measure on $\mathbb{R}^{d+1}_+\times\mathbb{R}_+$ such that $\partial_v m$ is a signed Radon measure on $\mathbb{R}^{d+1}_+\times[0,L]$ with the following properties:
\begin{enumerate}
\item[i.](Kinetic equation) $f$ and $m$ verify (in distributional sense) equation
\begin{equation}
\label{part:2}
\partial_t f{}+{}A'\cdot\grad_x f{}={}\partial_v m.
\end{equation}
Moreover,
\begin{equation}
\label{entropy:control}
||\partial_v m||_{\mathbb{R}^{d+1}_+\times[0,L]}{}\leq{}\frac{2}{\e}\iint \eta_{\e}f_0\,dxdv;
\end{equation}

\item[ii.](Optimality) for a.e. $(x,t),$ $f$ is a minimizer of 
\begin{equation}
\label{min:theorem}
\min\left\{ \int \eta_\e(v)f(v)\,dv\,:\, f(v)\in[0,1],\,\int f\,dv{}={}\rho(x,t)\right\};
\end{equation}

\item[iii.](Equi-continuity) for a.e. $t>0,$ and any $\xi\in\mathbb{R}^d,$
\begin{equation}
\label{th:continuity}
\iint |f(x+\xi,t,v)-f(x,t)|\,dxdv {}\leq{}\iint|f_0(x+\xi,v)-f_0(x,v)|\,dxdv.
\end{equation}

%\item for every convex function  $\eta$ on $[0,M],$
%\[
%\partial_t \int \eta' f\,dv{}+{}\div_x \int \eta' A'f\,dv\leq 0,\quad \mathcal{D}'(\mathbb{R}^+_{x,t});
%\]

%\item[iii.] there is a parametrized, unit mass, signed measure $\mu_{x,t}$ on $\mathbb{R}^d_v,$ such that $\mu_{x,t}$ is a measure-valued solution of the equation in \eqref{SCL}:
%\[
%\partial_t\langle \rho,\mu_{x,t}\rangle{}+{}\div_x\langle A(\rho),\mu_{x,t}\rangle{}={}0,
%\]
%and $\mu_{x,t}$ is close to a delta mass concentrated at $\rho(x,t):$
%\[
%\mu_{x,t}{}={}\delta(v-\rho(x,t)){}+{}\mu^\e_{x,t},
%\]
%with 
%\[
%\mbox{\rm mass}|\mu^\e_{x,t}|,\,\mbox{\rm diam}(\supp \mu^\e_{x,t}){}\leq{} C\e,\quad \mbox{a.e. } (x,t)\in\mathbb{R}^{d+1}_+.
%\]
\end{enumerate}
\end{theorem}

Kinetic functions from theorem \ref{th:1} give rise to the approximate solutions of the conservation law \eqref{SCL}, as described in the next theorem.

\begin{theorem}
\label{th:2} 
For function $f$ from the previous theorem,  moments
\begin{equation}
\label{moments:2}
\rho(x,t){}={}\int f(x,t,v)\,dv,\quad
\phi(x,t){}={}\int A'(v)f(x,t,v)\,dv
\end{equation}
have the following properties.
\begin{enumerate}
\item[i.] $\rho,\phi\in L^\infty(\mathbb{R}^{d+1}_+)$ and  verify  (in distributional sense) conservation law
\begin{equation}
\label{part:1}
\partial_t \rho{}+{}\div_x \phi{}={}0.
\end{equation}
For any $\psi\in C^\infty_0(\mathbb{R}^d),$ $\int \rho(x,t)\psi(x)\,dx$ is continuous in $t$ and 
\[
\lim_{t\to0+}\int \rho(x,t)\psi(x)\,dx = \iint f_0(x,v)\psi(x)\,dvdx;
\]

\item[ii.] for any two pairs of values $(\rho(x,t),\phi_i(x,t))$ and $(\rho(y,\tau),\phi_i(y,\tau)),$ such that  $|\rho(x,t)-\rho(y,t)|\geq c_0\e,$ it holds:
\begin{equation}
\label{shock:speed:1}
\frac{\phi_i(x,t)-\phi_i(y,\tau)}{\rho(x,t)-\rho(y,\tau)}{}={}
\frac{A_i(\rho(x,t))-A_i(\rho(y,\tau))}{\rho(x,t)-\rho(y,\tau)}{}+{}O(\e),\quad i=1..d;
\end{equation}

\item[iii.](Limit to Kruzhkov's solution) As a function of $\e,$ $\rho$ converges the unique, entropy solution of the conservation law \eqref{SCL}, when $\e\to0.$
\end{enumerate}

\end{theorem}

\begin{remark}
The existence of an approximating pair $(\rho,\phi)$ with properties i. and ii. can also be established via a kinetic averaging lemma of G\'{e}rard \cite{Gerard}, using the estimate \eqref{entropy:control}, provided that the following  non-degeneracy condition holds:
\begin{equation}
\label{main:nondegeneracy}
\forall\, \sigma \in \mathbb{S}^{d-1},\quad \forall \xi\in\mathbb{R},\quad {\rm meas}\left\{ v\in [0,L]\;:\; A'(v)\cdot\sigma{}={}\xi\right\}{}={}0,
\end{equation}
where $\mathbb{S}^{d-1}$ is the unit sphere in $\mathbb{R}^d.$ This approach is independent of the particular structure of the right-hand side of \eqref{part:2} as $\partial_v m,$ or  $L^1$--contraction property \eqref{contract:1} that we use in the proof, both being characteristic properties of scalar conservation laws. 
\end{remark}

\subsection{Proof of theorem \ref{th:1}}

For a non-negative constant $\rho\in[0,L]$ consider a minimization problem
\begin{equation}
\label{2:min}
\min\left\{ \int \eta_\e(v)f(v)\,dv\,:\, f(v)\in[0,1],\,\int f\,dv{}={}\rho\right\}.
\end{equation}

In the next lemma $\id_A(v)$ stands for a characteristic function of set $A.$
\begin{lemma}
\label{lemma:1} Let $
n{}={}\lfloor \rho/\e\rfloor.$
The minimum in problem \eqref{2:min} equals
\[
\left\{
\begin{array}{ll}
\e\sum_{k=0}^{n-1}k+ \e n(\rho-n\e), & n\geq 1,\\
0, & n=0.
\end{array}
\right.
\]
It is achieved on minimizers
\[
f_{min}(v){}={}\id_{[0,n\e]}(v){}+{}\tilde{f}(v),
\]
where $\tilde{f}$ is an arbitrary function verifying conditions:
\begin{align}
&\tilde{f}(v)\in [0,1],\quad \forall v\in[0,L];\quad \supp\tilde{f}\subset [n\e,(n+1)\e];\\
\label{min:3}
&\int\tilde{f}\,dv{}={}\rho - n\e.
\end{align}
\end{lemma}
\begin{proof}
$\eta_\e(v)$ is a non-decreasing function. To minimize  functional $\displaystyle{\int \eta_\e f\,dv}$ one needs to pick $f$ that has all its mass as close to $v=0$ as possible, and is less than or equal $1.$ This shows the first statement. On interval $[n\e,(n+1)\e],$  a minimizer $f_min$ can be arbitrarily re-arranged without changing the value of its $\eta_\e$ moment.   This leads to the second part of the lemma.
\end{proof}

Given a kinetic density $f$ we  select a particular minimizer of \eqref{2:min} with $\rho=\int f\,dv$ in the following way.
If $\int_{(n+1)\e}^Lf\,dv > n\e - \int_0^{n\e}f\,dv,$
we set 
\begin{equation}
\label{case:1}
M_f(v) = \id_{[0,n\e+v_0]}(v) + f(v)\id_{(n\e+v_0,(n+1)\e)}(v),
\end{equation}
where $v_0\in(0,\e)$ is determined by the relation $\int M_f\,dv{}={}\int f\,dv.$ It is the smallest number such that 
\[
\int_0^{n\e+v_0}1-f\,dv{}={}\int_{(n+1)\e}^L f\,dv.
\]

If $\int_{(n+1)\e}^Lf\,dv \geq n\e - \int_0^{n\e}f\,dv,$
we set  
\begin{equation}
\label{case:2}
M_f(v) = \id_{[0,n\e]}(v) + f(v)\id_{(n\e,n\e+v_0)}(v),
\end{equation}
where  $v_0\in(0,\e)$ is uniquely determined 
as the smallest number such that 
\[
\int_0^{n\e}1-f\,dv{}={}\int_{n\e+v_0}^L f\,dv.
\]

This minimizer can be thought of as a rearrangement of mass $f$ obtained by shifting its pieces by to the locations with smaller values of $\eta_e(v).$

The key properties of the minimizer $f_{min}$ are listed in the next lemma.
\begin{lemma}
\label{lemma:2}
 Let $f$ be any function with values in $[0,1]$  and supported on $[0,L].$ For $M_f$, defined above
\begin{equation}
\label{control:1}
\int |f-M_f|\,dv \leq \frac{2}{\e}\int \eta_\e(v)(f-M_f)\,dv.
\end{equation}
For any non-decreasing function  $\eta,$ 
\begin{equation}
\label{convex:1}
\int \eta(v) (f(v)-M_f(v))\,dv\geq 0.
\end{equation}
For any two functions $f_1,$ $f_2$  with values in $\{0,1\}$ and supported on $[0,L],$
\begin{equation}
\label{contract:1}
\int |M_{f_1} - M_{f_2}|\,dv{}\leq{}\int |f_1 - f_2|\,dv,
\end{equation}
where $M_{f_1},M_{f_2}$ are the corresponding minimizers.
\end{lemma}
\begin{proof} 
Let $n$ be as in the previous lemma. Consider case \eqref{case:1}
\begin{multline}
\int |f-M_f|\,dv {}={}\int_0^{n\e+v_0} 1-f\,dv{}+{}\int_{(n+1)\e}^L f\,dv\\{}={}2\int_{(n+1)\e}^L f\,dv
{}\leq{}\frac{2}{\e}\int\eta_\e(v)(f-M_f)\,dv,
\end{multline}
where the last inequality holds since all mass of $f$ on interval $[(n+1)\e,L]$ has been removed from that interval.
Similarly, in case \eqref{case:2}
\begin{multline}
\int |f-M_f|\,dv {}={}\int_0^{n\e} 1-f\,dv{}+{}\int_{n\e+v_0}^L f\,dv\\{}={}2\int_{n\e+v_0}^L f\,dv
{}\leq{}\frac{2}{\e}\int\eta_\e(v)(f-M_f)\,dv.
\end{multline}

For  a non-decreasing function $\eta$, \eqref{convex:1} follows from the definition of $M_f.$

To prove \eqref{contract:1} it suffices to show that
\begin{equation}
\label{contract:2}
\int f_1 f_2\,dv{}\leq{}\int M_{f_1}M_{f_2}\,dv,
\end{equation}
since functions take only values $0$ or $1.$ Let $n_1,v_{1,0}$ and $n_2,v_{2,0}$ be the corresponding values of $n$ and $v_0$ from \eqref{case:1}, \eqref{case:2} for functions $f_1$ and $f_2.$

Consider the case $n_1> n_2$ first. Here
\[
\int M_{f_1}M_{f_2}\,dv {}={}\int_0^{(n_2+1)\e}M_{f_2}\,dv{}={}\int f_2\,dv\geq \int f_1 f_2\,dv.
\]
Next, consider the case $n_1=n_2$ ($=n$).
Suppose that representation \eqref{case:1} applies to both functions $f_1,f_2,$ and assume  $v_{1,0}\geq v_{2,0}.$
Then,
\[
\int M_{f_1}M_{f_2}\,dv {}\geq{}\int_{n\e+v_{1,0}}^{(n+1)\e}f_1f_2\,dv
+\int_0^{n\e+v_{1,0}}f_2\,dv {}+{}\int_{(n+1)\e}^L f_2\,dv{}\geq \int f_1 f_2\,dv.
\]
Suppose that representation \eqref{case:2} applies to both functions $f_1,f_2,$ and assume  $v_{1,0}\geq v_{2,0}.$
Then,
\[
\int M_{f_1}M_{f_2}\,dv {}\geq{} \int_0^{n\e} f_2\,dv{}+{}
\int_{n\e+v_{2,0}}^Lf_2\,dv{}+{}\int_{n\e}^{n\e+v_{2,0}}f_1f_2\,dv{}\geq{}\int f_1f_2\,dv.
\]
Suppose that \eqref{case:2} applies to function $f_1$ and \eqref{case:1} to $f_2.$ If $v_{1,0}\geq v_{2,0}$ then
\[
\int M_{f_1}M_{f_2}\,dv {}\geq{} \int_0^{n\e+v_{2,0}} f_1\,dv{}+{}
\int_{n\e+v_{1,0}}^Lf_1\,dv{}+{}\int_{n\e+v_{2,0}}^{n\e+v_{1,0}}f_1f_2\,dv{}\geq{}\int f_1f_2\,dv.
\]
If $v_{1,0}<v_{2,0}$ then
\[
\int M_{f_1}M_{f_2}\,dv {}\geq{} \int_0^L f_1\,dv{}\geq{}\int f_1f_2\,dv.
\]
The contraction property \eqref{contract:1} is proved now.
\end{proof}

Now we consider a discrete-time approximation, with time step $h>0$ and $t_n=nh,$ $n=0,1,2...$ Given $f_{n-1}(x,v)$ the next period kinetic function 
\[
f_{n}(x,v) = M_{\hat{f_{n}}},\quad \hat{f_n}(x,v) = f_{n-1}(x-A'(v)h),
\]
with $f_0$ being the initial data.
A continuous time approximate is defined as
\begin{equation}
\label{cont:approx}
f^h(x,v,t){}={}\left\{
\begin{array}{ll}
f_{n-1}(x-A'(v)(t-nh)), & t\in[(n-1)h,nh),\\
f_n(x,v), & t=nh.
\end{array}
\right.
\end{equation}
\begin{remark}
It can be easily seen that in dimension one, if initial data $f_0$ is such that $f_0(x,v)=1,$ for $0\leq v\leq k\e$ and $f_0(x,v)=0$ for $v>((k+1)\e$ then $f_n$ is evolved by simple translation with kinetic velocities $v,$ leading to dispersion effect. On the other hand if initial data, for example, has a form
	\[
	f_0(x,v) = \left\{
	\begin{array}{ll}
	\id_{[0,v_1]}(v),& x<0,\\
	\id_{[0,v_2]}(v), & x>0,
	\end{array}
	\right.
	\] 
with $v_1-v_2>\e$ and $A(v)=v$ (corresponding to Burger's equation) then $f_n$ evolves as a classical shock wave in a discrete-time approximation. 	
\end{remark}

The following properties of $f^h$ follow easily from its definition and properties established in lemma \ref{lemma:2}.
\begin{lemma}
\label{lemma:3}
It holds:
\begin{enumerate}
\item[i.] for any $(x,v,t),$ $f^h\in\{0,1\};$
\item[ii.] for any $(x,t),$ ${\rm supp }f^h\subset [0,L];$
\item[iii.] for any $t>0,$
\begin{equation}
\iint f^h(x,v,t)\,dvdx{}\leq{}f_0(x,v)\,dvdx;
\end{equation}
\begin{equation}
\iint \eta_\e(v)f^h(x,v,t)\,dvdx{}\leq{}\eta_\e(v)f_0(x,v)\,dvdx;
\end{equation}
\item[iv.] $f^h$ is a weak solution of the equation
\begin{equation}
\label{eq:discrete}
\partial_tf^h{}+{}A'(v)\cdot\grad_xf^h{}={}R^h,
\end{equation}
where
\begin{equation}
R^h{}={}\sum_{n=1}^\infty \delta(t-nh)(f_n(x,v) - f_{n-1}(x-A'(v)h));
\end{equation}
\item[v.] for any $t>0$ and any $\xi\in\mathbb{R}^d,$
\[
\iint |f^h(x+\xi,t,v)-f^h(x,t)|\,dxdv {}\leq{}\iint|f_0(x+\xi,v)-f_0(x,v)|\,dxdv.
\]
\end{enumerate}
\end{lemma}
Next, we estimate the interaction term $R^h$ in equation \eqref{eq:discrete}
\begin{lemma}For any $t>0,$
\[
\iint R^h\,dvdx{}\leq{}\sum_{n=1}^\infty \delta(t-nh)\iint|(f_n(x,v) - f_{n-1}(x-A'(v)h))|\,dxdv;
\]
and 
\[
\int_0^\infty \iint |R^h|\,dvdxdt{}\leq\frac{2}{\e}\iint\eta_\e(v)f_0(x,v)\,dvdx.
\]
\end{lemma}
\begin{proof}
The first inequality is obvious. Using equation \eqref{eq:discrete} we find that
\[
\sum_{n=1}^\infty \iint \eta_\e(v)(f_n(x,v)-f_{n-1}(x-A'(v)h,v))\,dxdv{}\leq{} \iint \eta_\e(v)f_0(x,v)\,dxdv.
\]
Since $f_n{}={}M_{f_{n-1}(x-A'(v)h,v)},$ using inequality \eqref{control:1} we get
\[
\sum_{n=1}^\infty \iint |f_n(x,v)-f_{n-1}(x-A'(v)h,v))|,dxdv{}\leq{} \frac{2}{\e}\iint \eta_\e(v)f_0(x,v)\,dxdv,
\] 
from which the second inequality of the lemma follows.
\end{proof}

With the information from the last two lemma, we consider compactness properties of $f^h$ as $h\to0.$ There is $f$ with a.e. values in $[0,1]$ and a signed Radon measure $\tilde{m}$ such that on a suitable subsequence $h_k\to0,$
\begin{align*}
f^{h_k}\to f\quad \mbox{\rm *-weakly in } L^\infty(\mathbb{R}^{d+1}_+\times[0,L]),\\
 R^{h_k}\to \tilde{m}\quad  \mbox{\rm *-weakly in } \mathcal{M}_{loc}(\mathbb{R}^{d+1}_+\times[0,L]),
\end{align*}
 for a.e. $t>0,$
\begin{align*}
\iint f(x,v,t)\,dvdx{}\leq{}\iint f_0(x,v)\,dvdx,\\
\iint \eta_\e(v)f(x,v,t)\,dvdx{}\leq{}\iint \eta_\e(v)f_0(x,v)\,dvdx,
\end{align*} 
and inequalities \eqref{entropy:control} and \eqref{th:continuity} hold.

Inequality \eqref{convex:1} implies that $\langle \tilde{m},\eta(v)\psi(x,t)\rangle{}\leq{} 0$ for any continuously non-decreasing function $\eta,$ and any non-negative $\psi\in C^\infty_0(\mathbb{R}^{d+1}_+).$ Thus, $\tilde{m}{}={}\partial_v m$ for a non-negative Radon measure.

To complete the proof of theorem \ref{th:1} it remain to establish
\eqref{min:theorem}. For that we first show that $v$--moments of $f^h$ are compact in $L^p$ norms.

\begin{lemma}
Let $\omega(v)$ be a measurable, bounded function on $[0,L].$ Then, the set of moments
\[
\left\{ \int \omega(v)f^h(x,v,t)\,dv\right\} \quad \mbox{pre-compact in } L^p_{loc}(\mathbb{R}^{d+1}_+),\,p\in[0,+\infty).
\]	
\end{lemma}
\begin{proof}
 Denote by $\rho^h_{\omega}{}={}\int \omega(v)f^h(x,v,t)\,dv.$
$\rho^h_{\omega}$ is bounded in $L^\infty(\mathbb{R}^{d+1}_+).$ It follow from part v. of lemma \ref{lemma:3} that for any $\xi\in\mathbb{R}^d,$ and any $T>0,$ and $p\in[1,+\infty),$
\[
\| \rho^h_\omega(x+\xi,t)-\rho^h_{\omega}(x,t)\|_{L^\infty((0,T); L^p(\mathbb{R}^d))}\to0,\quad |\xi|\to0,
\] 
uniformly in $h.$ It follow from equation \eqref{eq:discrete} that for any $T>0$ and $p\in[0,+\infty),$ 
\[
\left\{\partial_t \rho^h_\omega\right\} \quad \mbox{bounded in }\mathcal{M}((0,T);L^p(\mathbb{R}^d)) {}+{} L^\infty((0,T); W^{-1,p}_{loc}(\mathbb{R}^d)).
\]
Under these conditions, compactness lemma 5.1 of Lions \cite{LionsII} ensures that on a suitable sequence of values of $h\to 0,$ $(\rho^h_\omega)^2\to (\rho_\omega)^2$ in distributional sense, where $\rho_\omega$ is a limiting point of $\rho^h_\omega$ in *-weak topology of $L^\infty(\mathbb{R}^{d+1}_+).$ This implies the statement of the lemma. 
\end{proof}

A little bit more can be said about moments $\rho^h{}={}\int f^h(x,v,t)\,dv.$ Indeed, 
\[
\{\partial_t\rho^h \} \quad \mbox{bounded in }L^\infty((0,T); W^{-1,p}_{loc}(\mathbb{R}^d)),\,p\in[1,\infty).
\]
Thus, $\rho^h$ converges  for a limiting point $\rho,$ in $ C([0,T]; W^{-1,p}_{loc}(\mathbb{R}^d)).$ This shows, in  particular, that  $\rho(x,0){}={}\int f_0(x,v)\,dv.$

We consider the moments of $f^h$ from the set $\omega\in\{1,\eta_\e(v),A_1(v),..,A_d(v)\}$ and select a sequence $h=h_k\to0$ on which $f^h$ and $\rho^h_\omega$ converge in the topologies described above to their limiting values.

 To finish the proof of theorem \ref{th:1} it remain to establish \eqref{control:1}.

 Consider a piece-wise constant in time interpolation of functions $f_n(x,v):$
\[
\tilde{f}^h(x,t,v){}={}
f_{n-1}(x,v),\quad  t\in[(n-1)h,nh),\quad n=1,2,3...
\]
For any  $\psi\in C^\infty_0(\mathbb{R}^{d+1}_+\times[0,L]),$ using the definition of function $f^h$ we find that
\begin{multline*}
\iiint \psi(f^h-\tilde{f}^h)\,dvdxdt\\ {}={}
\sum_{n=1}^\infty \int_{(n-1)h}^{nh}(\psi(x,t,v)-\psi(x+A'(v)(t-(n-1)h),t,v)f_{n-1}\,dvdxdt{}={}O(h).
\end{multline*}
Thus  $\tilde{f}^h$ converges the same $f$ in *-weak topology of $L^\infty(\mathbb{R}^{d+1}_+\times[0,L]).$ Furthermore, there is another sequence $\hat{f}^h$ constructed by taking suitable convex linear combinations of a finite number of elements of $\{\tilde{f}^h\}$ that converges to $f$ in $L^p_{loc}()$ and a.e. $(x,t,v).$

Let $\hat{\rho}^h=\int \hat{f}^h\,dv.$ For each $(x,t),$  $\hat{f}^h(x,t,v)$ is a minimizer of the problem \eqref{min:theorem} with $\rho=\hat{\rho}^h(x,t).$ Since this problem depends continuously on the value of the constraint $\hat{\rho}^h$ and the latter converges a.e. $(x,t)$ to $\rho(x,t),$ then the limit of the minimizers $\tilde{f}^h$ is a minimizer corresponding to $\rho.$

\subsection{Proof of theorem \ref{th:2}}

Part i. of the theorem \ref{th:2} was established in proving theorem \ref{th:1}. Part ii. follows from from \eqref{min:theorem} and lemma \ref{lemma:1}. Indeed, let $\rho,$ and $\phi$ be given by \eqref{moments:2}, and $(x,t)$ is such that  $f(x,t,\cdot)$ is the minimizer of \eqref{min:theorem}. Let $n$ and $\tilde{f}$ be as in lemma \ref{lemma:1}. We can write for any $i=1..d,$
\begin{multline*}
\phi_i(x,t){}={}\int A_i'(v)f(x,t,v)\,dv{}={}
A_i(\rho(x,t)){}+{}\int_{n\e}^{(n+1)\e} A_i'(v)\left(\tilde{f}-\id_{[0,\rho]}(v)\right)\,dv\\
{}={}A_i(\rho(x,t)){}+{}\int_{n\e}^{(n+1)\e} \left( A_i'(v)-A_i'(n\e)\right)\left(\tilde{f}-\id_{[0,\rho]}(v)\right)\,dv\\
{}={}A_i(\rho(x,t)){}+{}O(\e^2),
\end{multline*}
which establishes \eqref{shock:speed:1}.

To show part iii. of the theorem we consider the sequence of kinetic functions $f^\e$ and their moments $\rho^e=\int f^\e\,dv,$ $\phi^\e_i{}={}\int A_i'(v)f^\e\,dv$ from theorem \ref{th:1} in the limit $\e\to0.$ Given the uniform bounds on the sequence  $f^\e$, continuity estimate \eqref{th:continuity} and equation one can repeat the arguments of the proof of theorem \ref{th:1} to establish that $v$--moments of $f^\e$ are pre-compact in $L^p_{loc}(\mathbb{R}^{d+1}_+)$ and (one a subsequence) converge to a pair $(\rho,\phi)$ -- a solution of \eqref{part:1}, while $f^\e$ itself converges weakly to a function that $f$ that verifies the kinetic equation \eqref{part:2} (but not \eqref{entropy:control}) and which a.e. $(x,t)$ is a minimizer of problem \eqref{min:theorem} with function $\eta(v)=v,$ in place of $\eta_\e.$ This new problem
\[
\min\left\{ \int \eta(v)f(v)\,dv\,:\, f(v)\in[0,1],\,\int f\,dv{}={}\rho(x,t)\right\}
\]
has a unique minimizer in the form $f(x,v,t){}={}\id_{[0,\rho(x,t)]}(v).$
Thus, $\phi(x,t) {}={}A(\rho(x,t))$ a.e. $(x,t)$ and $\rho$ is a unique entropy solution of the conservation law \eqref{SCL}. The uniqueness implies that the sequence $\rho^\e$ converges to $\rho$ in the limit of $\e\to0.$

\bibliographystyle{siamplain}

\bibliography{references}

\end{document}